\documentclass[10pt,reqno]{amsart}
\usepackage[top=1in, bottom=1in, left=1in, right=1in]{geometry}
\usepackage{graphicx}
\usepackage[dvipsnames]{xcolor}
\usepackage{amssymb, amsthm, mathrsfs, amsmath, amsfonts,dsfont, mathtools}
\usepackage{epstopdf,color}
\usepackage{enumerate, enumitem}
\usepackage{lipsum}
\usepackage{hyperref}

\protected\def\ignorethis#1\endignorethis{}
\let\endignorethis\relax

\usepackage{fancyvrb}

\newenvironment{lcases}
  {\left\lbrace \begin{aligned}}
  {\end{aligned} \right.}

\newtheorem{thm}{Theorem}[section]

\newtheorem{prop}{Proposition}[section]
\newtheorem{cor}{Corollary}[section]

\theoremstyle{definition}

\newtheorem{rem}{Remark}[section]

\setenumerate[0]{label=(\roman*)}

\DeclareMathOperator{\Tr}{tr}

\def\R{\mathbb{R}}

\def\Z{\mathbb{Z}}

\def\C{\mathbb{C}}
\def\W8{W_{\infty}}
\def\d{\delta}

\def\d*{\delta_*}

\def\w{\omega}
\def\div{{\rm div}}
\def\supp{{\rm supp}}
 
\def\phi{\varphi}

\def\divr{ \div_{\rho^c}}
\def\dsg{d\sigma^c}
\def\dr{d\rho^c}
\def\us{ \tilde u}

\def\Emgs{\mathbb{E}_{\mu_{\sigma^c,\gamma}}}

\begin{document}

\title{On a non-periodic modified Euler equation: existence and quasi-invariant measures }
\author[Ana Bela Cruzeiro and Alexandra Symeonides]{Ana Bela Cruzeiro (1) and Alexandra Symeonides (2)}
 \address{
(1) GFMUL and Dep. de Matem\'atica do Instituto Superior T\'ecnico, Univ. de Lisboa \hfill\break\indent 
		Av. Rovisco Pais, 1049-001 Lisboa, Portugal \hfill\break\indent
			}
 \email{abcruz@math.tecnico.ulisboa.pt}

\address{
 (2)  GFMUL and Dep. de Matem\'atica da Faculdade de Ci\^encias, Univ. de Lisboa\hfill\break\indent 
		Campo Grande, 1749-016 Lisboa, Portugal \hfill\break\indent
			}
 \email{asymeonides@fc.ul.pt}

\begin{abstract}\noindent
We consider a modified Euler equation on $\mathbb R^2$. We prove existence of weak global solutions for bounded (and fast decreasing at infinity) initial conditions and construct Gibbs-type measures on function spaces which are quasi-invariant for the Euler flow. Almost everywhere with respect to such measures (and, in particular, for less regular initial conditions),  the flow is shown to be also globally defined.
\end{abstract}

\maketitle

\tableofcontents

\section{Introduction}

The Cauchy problem for the Euler equation is a challenging problem in nonlinear partial differential equations. Local existence of smooth solutions was proved by Lichtenstein in 1925 \cite{Lich}. In two-dimensions and in bounded domains existence, uniqueness and global regularity were shown when the initial vorticity is bounded, by Yudovich (1963) \cite{Y}. Solutions with initial data of finite energy were studied also by Kato \cite{Kato} and Bardos \cite{Bard}, among others. There is an extensive literature about local solutions of Euler equations, but much less is known about global ones. The only known results to the authors are due to DiPerna and Majda \cite{dPM} concerning very weak solutions and a recent work  \cite{BH}  dealing with special function spaces which allow for unbounded vorticities.

The least action principle on the diffeomorphisms group (Arnold \cite{Arn}, Ebin-Marsden \cite{EM}, more recently Brenier \cite{Bre})  is a different approach, that studies the Lagrangian problem for the position and not directly the Cauchy problem for the velocity field.

There is also the statistical approach to this type of equations, that consists in defining a priori invariant (or quasi-invariant) measures for the flow and using such measures to prove existence starting (almost everywhere) in the support of the measures. These supports are in general spaces of irregular functions. With respect to this approach, we mention \cite{AC} for the case of the periodic two-dimensional Euler equation. Recently, in \cite{CSy} we have obtained by these methods local solutions in the plane.

In this work, we consider a modification of the Euler equation (c.f.\eqref{Euler1}) involving the pressure term, which allows us to use the Ornstein-Ulhenbeck operator instead of the Laplacian in the vorticity equation and to use Sobolev spaces with respect to Gaussian measures. For this modified equation we first look for weak solutions starting with bounded functions. Then we construct quasi-invariant Gibbs-type measures and define global solutions of the equation for less regular initial conditions (in the support of such probability measures).

\addtocontents{toc}{\protect\setcounter{tocdepth}{1}}



\section{The modified Euler equation}

We present a different Euler equation, with a modification of the pressure contribution, namely

\begin{equation}\label{Euler1}
\frac{\partial \us}{\partial t} + (\us \cdot \nabla)\us =-\nabla p +cxp, \qquad \div \us=0
\end{equation}
where $\us:\R\times\R^2\to\R^2$ denotes the time dependent velocity field, $p:\R\times\R^2\to \R$  may depend on $c$ and $c$ is a fixed parameter in $(0,1)$.

After the  change of variables
\begin{equation}\label{change_var}
u(t,x)=\sigma^c(x) \us(t,x), \qquad t\in \R,~ x\in \R^2
\end{equation}
where  $\sigma^c(x)=\frac{1}{2\pi} e^{-\frac{c|x|^2}{2}}$ denotes a Gaussian density in $\R^2$, the equation reads,

\begin{equation}\label{Euler2}
\frac{\partial u}{\partial t} + (u \cdot \nabla)(\rho^c u) =-\nabla(\sigma^c p), \qquad \divr u=0,
\end{equation}
where $\rho^c(x):=({\sigma^c})^{-1}(x)=2\pi e^{\frac{c|x|^2}{2}}$ and $\divr u$ is defined by
$$
\int_{\R^2} \divr u f \dr = -\int_{\R^2} u \cdot \nabla f \dr, \qquad \forall f\in \mathcal{C}^1_c
$$
(for simplicity, we use the notation $\dr =\rho^c dx$).
  We assume that the initial condition for \eqref{Euler2} is defined by $u_0=\sigma^c \us_0$, where $\us_0$ is the initial data for \eqref{Euler1}, and that $\us$ and $u$ vanish sufficiently rapidly at infinity.
 
As we will see below, this change of variables allows us to study the equations in $L^2_{\sigma^c}(\R^2)$, the space of real-valued functions that are square integrable with respect to the measure $\sigma^c dx$.

\subsection*{Hermite polynomials and Gaussian Sobolev spaces} 
We recall the definition of the $k$-th order Hermite polynomial on $\R^2$
$$
H_k^c(x):=\Pi_{i=1,2} H_{k_i}^c(x_i), \qquad k\in \mathbb{Z}^2, ~k\geq 0
$$
where
$$
H_{k_i}^c(x_i) = \frac{1}{\sqrt{c^{k_i}}\sqrt{k_i !}}e^{\frac{cx_i^2}{2}}\frac{\partial^{k_i}}{\partial x_i^{k_i}}e^{-\frac{cx_i^2}{2}}, \qquad i=1,2
$$
denotes the one-dimensional Hermite polynomial of order $k_i$. We write
$$
H_k^c(x)= \frac{1}{\sqrt{c^{|k|}}\sqrt{k !}}e^{\frac{c|x|^2}{2}}D^k e^{-\frac{c|x|^2}{2}}
$$
where $|k|=k_1+k_2$, $k!=k_1 ! k_2! $ and $D^k=\frac{\partial^{k_1}}{\partial x_1^{k_1}}\frac{\partial^{k_2}}{\partial x_2^{k_2}} $. It is well known that the collection $\{H_k^c(x)\}_{k\geq 0}$ forms an orthonormal basis for $L^2_{\sigma^c}(\R^2)$. Moreover, the Hermite polynomials are eigenfunctions for the Ornstein-Uhlenbeck operator, $L^c: L^2_{\sigma^c}(\R^2) \to L^2_{\sigma^c}(\R^2)$, defined by 
$$
L^c \phi = \Delta \phi - cx\cdot \nabla \phi.
$$
We have 
$$
L^c H_k^c(x) = -c|k| H_k^c(x), \qquad \forall k\geq 0 .
$$
We recall some properties of the Hermite polynomials that we will use below. For the one-dimensional Hermite polynomials we have 
\begin{enumerate}
\item { \it Differentiation formula}:
\begin{equation}\label{diff_formula1D}
\frac{d}{dx}H_n^c(x)=-\frac{\sqrt{n}}{c}H_{n-1}^c(x);
\end{equation}
\item {\it Recursive relation}:
\begin{equation}\label{rec_relation1D}
\sqrt{n+1} H_{n+1}^c(x) - x H_n^c(x) + \sqrt{n} H_{n-1}^c(x)=0; 
\end{equation}
\item {\it Product formula}:
\begin{equation}\label{prod_formula1D}
H_n^c(x) H_m^c(x)= \sum_{r\leq n \wedge m} \Theta(n,m,r)H_{n+m-2r}^c(x),
\end{equation}
where
\begin{equation}\label{bin_coeffs1D}
\Theta(n,m,r)= \left[ \binom{n}{r} \binom{m}{r} \binom{n+m-2r}{n-r}\right]^{1/2}.
\end{equation}
\end{enumerate}
\begin{rem}
The first property is well-known, see for example \cite{MAKL}. The product formula can be found in \cite{EMOT} (pag. 195, equation (37)); here the formula is stated for the ``physicists'' Hermite polynomials, $\mathcal{H}_n(x)= (-1)^n e^{x^2}\frac{d^n}{dx^n}e^{-x^2}$. From the relation $H_n^c(x)=\frac{2^{-\frac{n}{2}}(-1)^n}{\sqrt{n!}}\mathcal{H}_n(\sqrt{\frac{c}{2}}x)$  we get equations \eqref{prod_formula1D}-\eqref{bin_coeffs1D}.
\end{rem}
Properties  \eqref{diff_formula1D} to \eqref{bin_coeffs1D} can be generalised for the two-dimensional Hermite polynomials. If $k\in \Z^2$ and $x\in \R^2$, we have
\begin{enumerate}
\item { \it 2D Differentiation formula}:
\begin{equation}\label{diff_formula}
\nabla H_k^c(x)=-\frac{1}{c}\left(\sqrt{k_1}H_{k_1-1}^c(x_1)H_{k_2}^c(x_2), \sqrt{k_2}H_{k_1}^c(x_1)H_{k_2-1}^c(x_2)\right);
\end{equation}
\item {\it 2D Recursive relation}:
\begin{align}\label{rec_relation}
& \mbox{ for } i=1,2\mbox{ and }j\neq i, \nonumber \\
& \sqrt{k_i+1} H_{k_i+1}^c(x_i)H_{k_j}^c(x_j)  - x_i H_{k_i}^c(x_i)H_{k_j}^c(x_j) + \sqrt{x_i} H_{k_i-1}^c(x_i)H_{k_j}^c(x_j)=0;
\end{align}
\item {\it 2D Product formula}:
\begin{align}\label{prod_formula}
H_k^c(x) H_h^c(x)& = \sum_{\substack{r_1\leq k_1 \wedge h_1 \\ r_2\leq k_2 \wedge h_2 }} \Theta(k_1,h_1,r_1)\Theta(k_2,h_2,r_2)H_{k_1+h_1-2r_1}^c(x_1)H_{k_2+h_2-2r_2}^c(x_2) \\
& = \sum_{r\leq k \wedge h} \tilde \Theta(k,h,r)H_{k+h-2r}^c(x), \nonumber
\end{align}
where
$$
 \tilde \Theta(k,h,r) := \Pi_{i=1,2}\Theta(k_i,h_i,r_i)
$$
and $\Theta$ is defined in \eqref{bin_coeffs1D}.
\end{enumerate}

Also consider for all $\beta \in \R $ the function spaces
$$
\mathcal{H}^\beta_{\sigma^c}(\R^2) = \left\{ v: \R^2 \to \R, v \in L^2_{\sigma^c}(\R^2) ~: ~ (I- L^c)^{\beta/2} v\in L^2_{\sigma^c}(\R^2)\right\};
$$
for $\beta$ negative or non-integer the operator $L$ is understood as a pseudo-differential operator in the Gaussian space of square integrable functions. 
The Sobolev spaces $\mathcal{H}^\beta_{\sigma^c}(\R^2)$ may be identified with the complex spaces
$$
H^\beta_{\sigma^c}(\R^2)=\left\{ v=\sum_{k \geq 0} v_k H_k^c ~ : ~ \sum_{k \geq 0}(1+ c|k|)^\beta |u_k|^2<+\infty \right\}.
$$
These are Hilbert spaces with inner products given by
$$
<u,v>_{\beta,\sigma^c}= \sum_{k \geq 0}(1+ c|k|)^\beta u_k \bar v_k.
$$
By $\| \cdot \|_{\beta,\sigma^c}$ we denote the norm of $H^\beta_{\sigma^c}(\R^2)$ for all $\beta\in \R$.

\section{The vorticity equation}

As usual the vorticity equations are obtained by taking the ``curl" of equation \eqref{Euler2}.
We have 
\begin{equation*}
\nabla^\perp \cdot [(u \cdot \nabla)(\rho^c u) ] = \nabla^\perp \cdot [(\sigma^c \us \cdot \nabla)\us ] = \sum_{i,j=1,2} \partial_i^\perp \sigma^c \tilde u_j \partial_j \tilde u_i + \sigma^c \partial_i^\perp \tilde u_j \partial_j \tilde u_i + \sigma^c \tilde u_j\partial_i^\perp \partial_j \tilde u_i,
\end{equation*}
where 
$$
\sum_{i,j=1,2} \partial_i^\perp \sigma^c \tilde u_j \partial_j \tilde u_i =  \sum_{i,j=1,2} -cx_i^\perp \sigma^c \tilde u_j \partial_j \tilde u_i = \sum_{i,j=1,2} -\sigma^c \tilde u_j \partial_j(c x_i^\perp \tilde u_i)
$$
and
$$
\sum_{i,j=1,2} \sigma^c \partial_i^\perp \tilde u_j \partial_j \tilde u_i = \sum_{i,j=1,2} \sigma^c  \partial_i \tilde u_i \partial_j^\perp \tilde u_j = 0,
$$
since $\div \tilde u = \sum_{i=1,2}\partial_i \tilde u_i= 0$. Also we have
$$
\nabla^\perp \cdot \nabla (\sigma^c p)=0.
$$ 
Moreover, since $\div \us =0$ we known that there exists a real-valued function $\phi: \R \times \R^2 \to \R$ such that $\us= \nabla^\perp \phi$ and $u=\sigma^c \nabla^\perp \phi$. Thus we have
$$
\nabla^\perp \cdot \us= \sigma^c \Delta \phi \quad \mbox{ and }\quad \nabla^\perp \cdot u= \sigma^c L^c \phi,
$$
from which follows that the vorticity equation can be written as 
\begin{equation}\label{vort_eq1}
\frac{\partial }{\partial t}\sigma^c L^c\phi = -(\sigma^c \nabla^\perp \phi \cdot \nabla) L^c\phi, 
\end{equation}
or equivalently as
\begin{equation}\label{vort_eq2}
\frac{\partial}{\partial t} L^c\phi = -(\nabla^\perp \phi  \cdot \nabla) L^c \phi.
\end{equation}
In particular, we observe that the quantity $L^c \phi$ is conserved along the particle trajectories with velocity $\us$, that we denote by $\Phi_t$, that is
\begin{equation}\label{conservation}
L^c \phi (t, x) = L^c \phi (0, \Phi_{-t} (x)), \qquad  t\in \R, ~x\in\R^2.
\end{equation}
Indeed, by definition of particle trajectories, we have
\begin{align}\label{ODE_Banach}
\frac{d}{dt} \Phi_t(x) & = \rho^c u (\Phi_t(x), t) \\
\Phi_0(x) & = x, \nonumber
\end{align}
thus for all $x\in\R^2$
\begin{align*}
\frac{d}{dt} L^c\phi(t, \Phi_t(x))&= \frac{\partial}{\partial_t} L^c \phi (t, \Phi_t(x)) + \frac{d}{dt} \Phi_t(x) \cdot \nabla L^c \phi (t, \Phi_t(x)) \\
&= \frac{\partial}{\partial_t} L^c \phi (t, \Phi_t(x)) + \rho^c u ( t, \Phi_t(x)) \cdot \nabla L^c \phi (t, \Phi_t(x))=0,
 \end{align*}
where the last equality follows from \eqref{vort_eq2}.

The $L^p$-norms of $L^c \phi$ are conserved for all $p\in \{1,2, \ldots ,\infty\}$; indeed for any $f$ measurable function
\begin{align*}
\frac{d}{dt}\int_{\R^2}f(L^c\phi(t,\Phi_t(x)))dx & =  \int_{\R^2}f'(L^c\phi(t,\Phi_{t}(x)))[\frac{\partial}{\partial_t}L^c\phi(t,\Phi_t(x))+ \rho^c u(t, \Phi_t(x))\cdot \nabla L^c\phi(t, \Phi_t(x)) ]dx=0.
\end{align*}
 For $p=2$, we directly prove the statement 
\begin{align*}
\frac{1}{2}\frac{d}{dt}\| L^c \phi \|^2_{L^2} & = -\int (\nabla^\perp \phi  \cdot \nabla) L^c \phi L^c \phi dx  =  \int \div (\nabla^\perp  \phi L^c \phi ) L^c \phi dx \\
& = \int (\nabla^\perp  \phi \cdot \nabla) L^c \phi L^c \phi dx = 0.
\end{align*}

\subsection*{Existence and uniqueness}\label{weak_sol}
In this section we look for pointwise  solutions of equations \eqref{Euler2}. By equation \eqref{conservation}, we obtain weak solutions of
\begin{align*}
\frac{\partial}{\partial t} L^c\phi & = -(\nabla^\perp \phi  \cdot \nabla) L^c \phi,
\end{align*}
if we are able to solve the associated ODE for the particle trajectories
\begin{align*}
\frac{d}{dt} \Phi_t(x) & = \rho^c u (\Phi_t(x), t) \\
\Phi_0(x) & = x. \nonumber
\end{align*}
If we define by $\w$ the vorticity of $u$, that is
$$\w = \nabla^\perp \cdot u = \sigma^c L^c \phi, $$
we have 
$$\rho^c u = K_{L^c} * \rho^c \w ,$$
where $K_L(x,y)$ denotes the orthogonal gradient of $G_{L^c}$, that in turn denotes the Green's function for the Ornstein-Uhlenbeck operator $L^c$ on $\R^2$. We consider initial data $\rho^c \w_0 \in L^1(\R^2)\cap L^\infty(\R^2)$. 

In order to compute $G_{L^c}$, we consider the operator $L^c= \Delta -cx \cdot \nabla$ as a perturbation of $\Delta$, thus we write $G_{L^c}$ in terms of $G$, where $G$ denotes the Green's function for $\Delta$ in $\R^2$. It is well-known that $G(x,y)=\frac{1}{2\pi}\ln |x-y|$. By definition
$$
L^c G_{L^c}(x,y)= \delta(x-y)= \Delta G(x,y)
$$
in the sense of distributions, that is
$$
\forall f \in L^\infty, \quad \int L^c G_{L^c}(x,y)f(y)dy= \int \delta(x-y)f(y)dy= \int \Delta G(x,y) f(y)dy.
$$
Now the idea is to apply $\Delta^{-1}$ to both members of the latter expression. Since $\rho^c u = K_{L^c} * \rho^c \w $, we  use $\Delta^{-1}$ in $L^1_{\rho^c}$. It is easy to check that 
\begin{equation}\label{G^r}
G^{\rho^c}(x,y)=\sigma^c(y)G(x,y)
\end{equation}
and 
\begin{equation}\label{G_L^r}
G^{\rho^c}_{L^c} (x,y)=\sigma^c(y)G_{L^c}(x,y)
\end{equation}
where $G^{\rho^c}$ and $G^{\rho^c}_{L^c}$ denote respectively the Green's functions for $\Delta$ and $L^c$ in $L^1_{\rho^c}$.
Hence, we get
\begin{equation}\label{GLr}
G_{L^c}^{\rho^c} = G^{\rho^c} + G^{\rho^c}* x\cdot \nabla_x G^{\rho^c}_{L^c}
\end{equation}
where this should be understood as 
$$
\forall f \in L^\infty, \quad \int G_{L^c}^{\rho^c} (x,y)f(y)dy= \int  G^{\rho^c}(x,y) f(y)dy+ \int\int G^{\rho^c}(x,z)z\cdot \nabla_z G_{L^c}^{\rho^c} (z,y)f(y)dzdy.
$$
Using iteratively equation \eqref{GLr}, we get an expression for $\rho^c u = K_{L^c} * \rho^c \w $ such that 
\begin{align*}
|\rho^c(x) u(x,t)| &= |\nabla^\perp \phi(x,t)| \leq \| \rho^c \w \|_{L^\infty} \left\{ \frac{1}{2\pi} \int \frac{\dsg(y)}{|x-y|} + \left(\frac{1}{2\pi}\right)^2 \int \int \frac{c|x_1|\dsg(x_1)\dsg(y)}{|x-x_1||x_1- y|} + \dots  \right. \\
& \left. +  \left(\frac{1}{2\pi}\right)^n \int \cdots \int \frac{c^{n-1}|x_1| \cdots |x_{n-1}|\dsg(x_1)\cdots \dsg(x_{n-1})\dsg(y)}{|x-x_1|\cdots|x_{n-1}- y|} + \dots  \right\}.
\end{align*}
The $n$-th term of the previous expansion is smaller than $\frac{\sqrt{2\pi}}{2}\left(\frac{1}{2\pi}\right)^nc^{n-1}$, thus we have $|\rho^c u|\lesssim  \| \rho^c \w \|_{L^\infty}$, where $\lesssim$ stands for less or equal up to a multiplicative constant.

\begin{thm}
Given $ \rho^c\w_0 \in L^1\cap L^\infty$, there exists $T>0$ such that equation \eqref{ODE_Banach} has a unique solution in $[-T,T]$ and $\rho^c \w \in L^\infty([-T,T]; L^1\cap L^\infty)$ is a weak solution for equation \eqref{vort_eq2}.
\end{thm}
\begin{proof}
By Osgood's theorem in Banach spaces (see \cite{S}), if $\rho^c u$ is a quasi-Lipschitz field, we obtain a unique solution for the Cauchy problem \eqref{ODE_Banach} in $[-T,T]$. For $x,x' \in \R^2$, we have
\begin{align*}
|\rho^c(x)u(x,t)- \rho^c(x')u(x',t)| & \leq \| \rho^c \w \|_{L^\infty} \left\{  \frac{1}{2\pi} \int \left| \frac{(x-y)^\perp}{|x-y|^2} - \frac{(x'-y)^\perp}{|x'-y|^2} \right| \dsg(y) + \right. \\
& \left.  \left(\frac{1}{2\pi}\right)^2 \int \int  \left| \frac{(x-x_1)^\perp}{|x-x_1|^2} - \frac{(x'-x_1)^\perp}{|x'-x_1|^2} \right| \frac{c|x_1|\dsg(x_1)\dsg(y)}{|x_1- y|} + \ldots \right\}.
\end{align*}
Below we follow Appendix 2.3 of \cite{MP} (where the case of a bounded domain is treated) to prove the quasi-Lipschitz continuity. Let $r:=|x-x'|$; for $r\geq 1$ the statement is a consequence of the previous computations, for $r<1$ we set $A:=\left\{ y\in \R^2 ~| ~|x-y|\leq 2r \right\}$ and  we write 
\begin{align*}
\int_{\R^2} \left| \frac{(x-y)^\perp}{|x-y|^2} - \frac{(x'-y)^\perp}{|x'-y|^2} \right| \dsg(y)&  = 
\int_{ A} \left| \frac{(x-y)^\perp}{|x-y|^2} - \frac{(x'-y)^\perp}{|x'-y|^2}\right| \dsg(y) + \int_{ A^c} \left| \frac{(x-y)^\perp}{|x-y|^2} - \frac{(x'-y)^\perp}{|x'-y|^2}\right| \dsg(y) .
\end{align*}
On one hand,
$$
\int_{ A} \left| \frac{(x-y)^\perp}{|x-y|^2} - \frac{(x'-y)^\perp}{|x'-y|^2}\right| \dsg(y) \leq \int_{ |x-y|\leq 2r} \left[ \frac{1}{|x-y|} + \frac{1}{|x'-y|}\right] \dsg(y) \lesssim r.
$$
On the other, choosing $x''$ to be a point belonging to the segment $x, x'$, for $y\in A^c$ we have $|x''-y|\geq \frac{1}{2}|x-y|$, thus
\begin{align*}
\int_{ A^c} \left| \frac{(x-y)^\perp}{|x-y|^2} - \frac{(x'-y)^\perp}{|x'-y|^2}\right| \dsg(y) & \lesssim r \int_{ A^c}\frac{1}{|x''-y|^2} \dsg(y) \\
& \lesssim r\left\{ \int_{2r< |x-y|<2} \frac{1}{|x-y|^2}\dsg(y) + \int_{\R^2} \dsg(y)\right\} \\
&\lesssim r \left\{ \int_{2r< |x-y|<2} \frac{1}{|x-y|^2}dy +1 \right\}.
\end{align*}
Computing the integrals we obtain
$$
\int_{\R^2} \left| \frac{(x-y)^\perp}{|x-y|^2} - \frac{(x'-y)^\perp}{|x'-y|^2} \right| \dsg(y) \lesssim \lambda(|x-x'|)
$$
where $\lambda$, defined by $\lambda(r)= r$, for $r \geq 1$ and by $\lambda(r)= r(1-\ln r)$, for $r<1$, is the modulus of continuity for $\rho^c u$. That is $\rho^c u$ is quasi-Lipschitz continuous and by Osgood's theorem there exists a unique flow given by $\Phi_t(x)=x+ \int_0^t \rho^c(x) u(s,\Phi_s(x))ds$ for $t\in [-T,T]$. From $ \rho^c(x)\w(x,t)=\rho^c(\Phi_{-t}(x))\w_0(\Phi_{-t}(x))$ and the assumptions we get $\rho^c \w \in L^\infty([-T,T]; L^1\cap L^\infty)$, which is  sufficient to verify the vorticity equation in the weak sense, that is
$$
\frac{d}{dt}\int \rho^c \w f dx = \int  \rho^c \w ( \nabla^\perp \phi \cdot \nabla f) dx, \qquad  \forall f \in C^1_0.
$$
\end{proof}



\section{Quasi-invariant measures}

On a probability space $(\Omega, \mathcal{F}, \mathbb{P})$ we let $\{ g_k \}_{k \in \Z^2}$ to be a sequence of independent and identically distributed random variables, where each $g_k$ is distributed as a standard, complex-valued Gaussian. We denote by $\tilde \lambda_k$ the eigenvalues of the Ornstein-Uhlenbeck operator $L^c$ on $L^2_{\sigma^c}$, that is $\tilde \lambda_k=-c|k|$ for all $k\geq 0$. For any given $n$, we consider the random variable 
$$
\Gamma^n_\gamma(\w, x):=\sum_{k\in \{\alpha_1, \ldots, \alpha_{d(n)}\}} \frac{g_k(\w)}{1-\tilde \lambda_k} H_k^c(x),
$$ 
whose law is given by 
$$
d\mu^n_{\sigma^c,\gamma}(\phi) \simeq \prod_{k\in \{\alpha_1, \ldots, \alpha_{d(n)}\}} \frac{\gamma(1+ c|k|)^2}{2\pi} e^{-\frac{\gamma}{2}(1+c|k|)^2 |\phi_k|^2}d\phi_k,
$$
for every $\phi(t,x)=\sum_{k \geq 0} \phi_k(t)H_k^c(x) \in L^2_{\sigma^c}$. The $\alpha_1, \ldots, \alpha_{d(n)}$ denote non-negative pairs of $\Z^{2}$.
In the limit when $n$  tends to infinity $\Gamma^n_\gamma(\w, x)$ converges pointwise to $\Gamma_\gamma(\w, x):= \sum_{k\geq 0} \frac{g_k(\w)}{1-\tilde \lambda_k} H_k^c(x)$ and we denote by $d\mu_{\sigma^c,\gamma}$ its law. For $\phi\in L^2_{\sigma^c}$,
\begin{align*}
d\mu_{\sigma^c,\gamma}(\phi) & \simeq \prod_{k\geq 0} \frac{\gamma(1+ c|k|)^2}{2\pi} e^{-\frac{\gamma}{2}(1+c|k|)^2 |\phi_k|^2}d\phi_k \\
& \simeq \frac{1}{Z_\gamma} e^{-\frac{\gamma}{2}\| (I-L^c) \phi \|_{L^2_{\sigma^c}}^2} \mathcal{D}\phi,
\end{align*}
thus $\mu_{\sigma^c,\gamma}$ is formally the Gibbs-type measure associated to the quantity $\frac{1}{2}\| (I-L^c)\phi \|_{L^2_{\sigma^c}}^2$.

For any $\gamma\in\R^+$, the triple $(H^{-\varepsilon}_{\sigma^c}, H^2_{\sigma^c},  d\mu_{\sigma^c,\gamma})$ is a complex abstract Wiener space for $\varepsilon>0$; $H^{-\varepsilon}_{\sigma^c}$ is the support of $\mu_{\sigma^c,\gamma}$ and $H^2_{\sigma^c}$ is the Cameron-Martin space. In particular; $\mathbb{E}_{\mu_{\sigma^c,\gamma}}(\phi_k \bar \phi_h)=\delta_{k,h}\frac{2}{\gamma (1+c|k|)^2}$, $\mathbb{E}_{\mu_{\sigma^c,\gamma}}(\phi_k )=0$, and $\mathbb{E}_{\mu_{\sigma^c,\gamma}}|\phi_k |^{2r}=\frac{2^r r!}{\gamma^r (1+c|k|)^{2r}}$.

Now we prove that the supports of the measures $\mu_{\sigma^c,\gamma}$ are not only spaces of very irregular functionals, but that in fact contain regular functions. Namely, $L^p_{loc}(\R^2)\subset \supp_{\mu_{\sigma^c, \gamma}}$ for every $p\in (2,10/3)$. We will use the so called ``dispersive bound'' for Hermite functions, firstly proved in dimension one by N. Burq, L. Thomann and N. Tzvetkov in \cite{BTT} and extended to other dimensions by A. Poiret in his Ph.D. thesis \cite{P}.

Below we denote by $h_k(x)$ the $k$-th order Hermite's function on $\R^2$, defined by $h_k(x)=h_{k_1}(x_1)\times h_{k_2}(x_2)$ for all $x\in \R^2$ and for all non-negative $k\in \Z^2$, where
$$
h_{k_i}(x_i)= \frac{(-1)^n 2^{-k_i /2}}{\sqrt{\sqrt{\pi} k_i!}}\frac{d^{k_i}}{d x^{k_i}}(e^{-x_i^2})e^{x_i^2/2}, \qquad i=1,2.
$$
It is well known that $h_k$ is an eigenfunction, with corresponding eigenvalue denoted by $\lambda_k^2$, for the harmonic oscillator $H:= -\Delta + |x|^2$ on $L^2(\R^2)$, that is $H h_k= \lambda_k^2 h_k$. The eigenvalues are $\lambda_k^2=\lambda_{k_1}^2+\lambda_{k_2}^2= (2k_1+1) + (2k_2+1)=2(|k|+1)$. For further details see \cite{P}. The relation between the Hermite's polynomials and the Hermite's functions is the following:
\begin{equation}\label{rel_hermite_poly_funct}
H_k^c(x)=(-1)^k  \sqrt{\sqrt{\pi} }~h_k\left(\sqrt{\frac{c}{2}} x\right)e^{\frac{c|x|^2}{4}}.
\end{equation}

The following result was proved in \cite{P}.
\begin{thm}[Dispersive bound]\label{disp_bound}
Let $d\geq 2$. There exists a constant $C>0$ such that for all $n,m$
\begin{equation*}
\| h_n h_m\|_{L^2(\R^d)} \leq C\times 
\begin{lcases}
 & \max(\lambda_n, \lambda_m)^{-\frac{2}{3} + \frac{d}{6}};\qquad  2\leq d\leq 4 \\
&  \max(\lambda_n, \lambda_m)^{-2 + \frac{d}{2}}; \qquad  d\geq 4 .
\end{lcases}
\end{equation*}
Moreover there exists a positive constant $C$ such that for all $n,m$
$$
\| h_nh_m\|_{L^{\frac{d+3}{d+1}}(\R^d)}\leq C \max(\lambda_n, \lambda_m)^{-\frac{1}{d+1}}.
$$
\end{thm}

In the two-dimensional case and for particular values of $p$, we show that the above result implies the following control over the $L^p$-norms of the Hermite's functions.

\begin{cor}\label{disp_bound_cor1}
For all $p\in(2,\frac{10}{3})$,
$$
\forall ~n, \qquad \| h_n\|_{L^p(\R^2)} \leq C \lambda_n^{(\theta-1)/6},
$$
where $\theta\in (0,1)$ is such that $\frac{1}{p}= \frac{\theta}{2} + \frac{1-\theta}{10/3}$.
\end{cor}
\begin{proof}
On one hand, from Theorem \ref{disp_bound} when $d=2$ and $n=m$, we get $\| h_n^2 \|_{L^{5/3}}\leq C \lambda_n^{-1/3}$. This implies
\begin{equation}\label{Lpbound}
\| h_n \|_{L^{10/3}(\R^2)}\leq C \lambda_n^{-1/6}.
\end{equation}
On the other hand, by H\"older's inequality,
$$
\| h_n\|_{L^p(\R^2)} \leq C\| h_n\|_{L^2(\R^2)} \| h_n\|_{L^{10/3}(\R^2)}^{1-\theta},
$$
where $\theta\in (0,1)$ is such that $\frac{1}{p}= \frac{\theta}{2} + \frac{1-\theta}{10/3}$. From the fact that $\{h_n\}_{n\geq 0}$ is an orthonormal basis for $L^2$ and from the bound \eqref{Lpbound} we conclude that
$$
\| h_n\|_{L^p(\R^2)} \leq C \| h_n\|_{L^{10/3}(\R^2)}^{1-\theta} \leq C \lambda_n^{(\theta-1)/6}.
$$
\end{proof}

Below we translate the above bounds in terms of Hermite's polynomials.

\begin{cor}\label{disp_bound_cor2}
For all $p\in(2,\frac{10}{3})$,
$$
\forall ~n, \qquad \| H_n^c\|_{L^p_{loc}(\R^2)} \leq C(p,c,R) \lambda_n^{(\theta-1)/6},
$$
where $\theta\in (0,1)$ is such that $\frac{1}{p}= \frac{\theta}{2} + \frac{1-\theta}{10/3}$ and where $R$ is a geometric constant that depends on each compact subset of $\R^2$ considered.
\end{cor}
\begin{proof}
Let $R>0$; by the relation \eqref{rel_hermite_poly_funct} we have
\begin{align*}
\| H_n^c \|_{L^p_(\{|x|<R\})} & = \pi^{1/4} \left( \int_{\{|x|<R\}} \left|h_n\left(\sqrt{\frac{c}{2}} x\right) \right|^p e^{\frac{cp}{4}|x|^2} dx \right)^{1/p} \\
& \leq C(p,c)   \left( \int_{\{|x|<\sqrt{\frac{c}{2}} R\}} \left|h_n(x) \right|^p e^{\frac{p}{2}|x|^2} dx \right)^{1/p} \\
& \leq C(p,c,R) \left( \int \left|h_n(x) \right|^p dx \right)^{1/p}.
\end{align*}
If $p\in(2,\frac{10}{3})$, by Corollary \ref{disp_bound_cor1} and since $R$ is arbitrary, we get
$$
\| H_n^c\|_{L^p_{loc}(\R^2)} \leq C(p,c,R)\lambda_n^{(\theta-1)/6}.       
$$
\end{proof}

Here we characterise the supports of the measures $\mu_{\sigma^c, \gamma}$ and in particular we see that they contain regular functions (and not only distributions).  

\begin{thm}\label{support_quasinv}
Let $\varepsilon>0$ and $p\in(2,\frac{10}{3})$; then
$$
\supp_{\mu_{\sigma^c, \gamma}}= H^{-\varepsilon}_{\sigma^c}(\R^2) \cap L^p_{loc}(\R^2).
$$
\end{thm}
\begin{proof}
As $\mu_{\sigma^c, \gamma}$ is the law of the random variable $\Gamma_\gamma$, its support is given by the spaces in which $\Gamma_\gamma(\w, \cdot)$ takes values $\mathbb{P}$-almost surely. For any arbitrary $R>0$ we have 
\begin{align*}
\left( \int_{\{|x|<R\}} \| \Gamma_\gamma(\w, x) \|^p_{L^2_\w} dx \right)^{1/p} & =  \left\{   \int_{\{|x|<R\}}  \left[ \mathbb{E}_{\mathbb{P}} \left(  \sum_{h,k}\frac{g_k(\w)\bar g_h(\w)}{(1-\tilde \lambda_k)(1-\tilde \lambda_h)} H_k^c(x)H_h^c(x) \right) \right]^{p/2} dx \right\}^{1/p}\\
& = \left(  \int_{\{|x|<R\}}   \left|  \sum_{k}\frac{|H_k^c(x)|^2}{(1-\tilde \lambda_k)^2 } \right|^{p/2} dx) \right)^{1/p}\\
&= \left\| \sum_{k}\frac{|H_k^c(x)|^2}{(1-\tilde \lambda_k)^2 } \right\|^{1/2}_{L^{p/2}(\{|x|<R\})}\\
& \leq \left(   \sum_{k}  \frac{1}{(1-\tilde \lambda_k)^2 }   \left\| H_k^c \right\|^{2}_{L^{p}(\{|x|<R\})} \right)^{1/2},
\end{align*}
which in turn, by Corollary \ref{disp_bound_cor2},
$$
\leq C(p,c,R)  \left(   \sum_{k}  \frac{\lambda_k^{-2\delta(k)}}{(1-\tilde \lambda_k)^2 }\right)^{1/2}   \leq C(p,c,R) \left(   \sum_{k}  \frac{1}{(1+c|k|)^{2+ \delta(k) } }\right)^{1/2} < +\infty
$$
with  $\delta(k)$  a strictly positive quantity. Moreover, for any $\varepsilon>0$ we have 
$$
\mathbb{E}_{\mu_{\sigma^c, \gamma}} \| \phi \|_{-\varepsilon, \sigma^c}^2 = \sum_k (1+c|k|)^{-\varepsilon} \mathbb{E}_{\mu_{\sigma^c, \gamma}} |\phi_k|^2= \frac{2}{\gamma}\sum_k \frac{1}{(1+c|k|)^{2+\varepsilon}} < +\infty.
$$
\end{proof}

\addtocontents{toc}{\protect\setcounter{tocdepth}{2}}



\section{The vorticity vector field}

Similarly to what was previously done for Euler equation in a compact domain  (c.f. \cite{AC}), we plan to write the vorticity equation,
$$
\partial_t L^c\phi = - (\nabla^\perp \phi \cdot \nabla) L^c\phi,
$$
as an infinite system of ordinary differential equations, using the orthonormal basis of $L^2_{\sigma^c}(\R^2)$ made of the Hermite polynomials $\{H_k^c\}_{k \in \Z^2}$. Let $\phi \in L^2_{\sigma^c}(\R^2)$ be such that $\phi(t,x)= \sum_{k\geq 0} \phi_k(t) H_k^c(x)$ for some $\phi_k : \R \to \C$ to determine. 
On one hand 
\begin{equation}\label{LHS}
\partial_t L^c\phi(t,x) = -c\sum_{k\geq 0}|k|\frac{d}{dt}\phi_k(t) H_k^c(x),
\end{equation}
on the other 
\begin{align*}
- (\nabla^\perp \phi \cdot \nabla) L^c\phi & =  c \sum_{p\geq 0}\sum_{\substack{ q\geq 0 \\ |q|<|p| }} (|p| -|q|) \phi_p \phi_q \nabla^\perp H_p^c \cdot \nabla H_q^c,
\end{align*}
since $\nabla^\perp H_p^c\cdot \nabla H_q^c= - \nabla H_p^c\cdot \nabla^\perp H_q^c$. By $p \geq 0$ we mean $p_i \geq 0$, for $i=1,2$. From Hermite polynomial's properties \eqref{diff_formula} and \eqref{prod_formula} we have
\begin{align*}
\nabla^\perp H_p^c \cdot \nabla H_q^c&= -\frac{1}{c^2}\sqrt{p_2 q_1}H_{p_1}^c(x_1)H_{p_2-1}^c(x_2)H_{q_1-1}^c(x_1)H_{q_2}^c(x_2) + \frac{1}{c^2}\sqrt{p_1 q_2}H_{p_1-1}^c(x_1)H_{p_2}^c(x_2)H_{q_1}^c(x_1)H_{q_2-1}^c (x_2)\\
& = -\frac{1}{c^2}\sqrt{p_2 q_1}\sum_{\substack{r_1 \leq p_1 \wedge q_1-1 \\ r_2 \leq p_2-1 \wedge q_2}}\Theta(p_1, q_1-1, r_1)\Theta(p_2-1, q_2, r_2)H_{p+q-1-2r}^c(x) \\
& \quad + \frac{1}{c^2}\sqrt{p_1 q_2} \sum_{\substack{r_1 \leq p_1-1 \wedge q_1 \\ r_2 \leq p_2 \wedge q_2-1}}\Theta(p_1-1, q_1, r_1)\Theta(p_2, q_2-1, r_2)H_{p+q-1-2r}^c(x)\\
&= \frac{1}{c^2}\sum_{\substack{|r| <|q| }}\left[-\sqrt{p_2 q_1}\Theta(p_1, q_1-1, r_1)\Theta(p_2-1, q_2, r_2) \right.\\
& \left. \quad+ \sqrt{p_1 q_2}\Theta(p_1-1, q_1, r_1)\Theta(p_2, q_2-1, r_2) \right] H_{p+q-1-2r}^c(x),
\end{align*}
where in the last equality we used $|q|<|p|$. We define $k=p+q-1-2r$, then $r=(p+q-1-k)/2$ and $0< |k| < 2|p|$; we get
\begin{align}\label{RHS}
 -(\nabla^\perp \phi \cdot \nabla) L^c\phi  = &\frac{1}{c}\sum_{p\geq 0}\sum_{0< |k|<2|p|}\sum_{\substack{ q\geq 0 \\ |q|<|p|}} (|p| -|q|)   A(p,q,k) \phi_p \phi_q H_k^c(x),
\end{align}
where
\begin{align}\label{A_pqk}
A(p,q,k)&:= \left[-\sqrt{p_2 q_1}\Theta(p_1, q_1-1,(p_1+ q_1-1 -k_1)/2)\Theta(p_2-1, q_2, (p_2+ q_2-1 -k_2)/2)  \right. \nonumber \\
& \left. \quad+ \sqrt{p_1 q_2}\Theta(p_1-1, q_1, (p_1+ q_1-1 -k_1)/2)\Theta(p_2, q_2-1, (p_2+ q_2-1 -k_2)/2) \right] .
\end{align}
Comparing equations \eqref{LHS} and \eqref{RHS}, the vector field $B^c$, corresponding to the equation 
$$
\frac{\partial }{\partial t}\phi(t,x)= B^c(\phi(t,x))
$$
where $\phi$ denotes the Euler stream-function, can be written as follows
\begin{equation}\label{B_before_permutation}
B^c(\phi)  = - \frac{1}{c^2}\sum_{p\geq 0}\sum_{0< |k|<2|p|}\sum_{\substack{ q\geq 0 \\ |q|<|p|}} \frac{1}{|k|}(|p| -|q|)   A(p,q,k)  \phi_p \phi_q  H_k^c(x) .
\end{equation}

\begin{rem}[Properties of A(p,q,k)]\label{prop_A}
For all non-negative $p,q,k$ the quantity $A(p,q,k)$ verifies the following properties, 
\begin{enumerate}
\item $A(p,q,k)=-A(q,p,k)$; \label{propA1}
\item $A(p,q,k)=0$, if $p_i>q_i+1 + k_i$ or $q_i>p_i+1+ k_i$ for some $i=1,2$; \label{propA2}
\item \label{propA3}
\begin{align*}
A(p,q,k)^2 & \lesssim \frac{p!q!k!}{(p+q-1-k)!^2}\left[ \frac{(p_1-q_1+1+k_1)(q_2-p_2+1+k_2) -(p_2-q_2+1+k_2)(q_1-p_1+1+k_1)}{(p-q+1+k)!(q-p+1+k)!} \right]^2 \\
&=\frac{p!q!k!}{(p+q-1-k)!^2} \left[ \frac{1}{(p_2-q_2+ 1 +k_2)(q_1-p_1+1+k_1)} - \frac{1}{(p_1-q_1+ 1 +k_1)(q_2-p_2+1+k_2)}  \right]^2    \\
& ~~~\times \frac{1}{(q-p+k)!^2 (p-q+k)!^2} \\
&\lesssim \frac{p!q!k!}{(p+q-1-k)!^2(q-p+k)!^2 (p-q+k)!^2}\lesssim \frac{p!q!}{(p+q-1-k)!^2  (q-p-1+k)!^2 k!}. 
\end{align*}
\end{enumerate}
\end{rem}

We can permute the series in the indices $k$ and $p$ that appear in the expression of $B^c$; moreover from property \ref{propA2} of $A(p,q,k)$, we deduce that the vorticity equation for \eqref{Euler2}  reads as 
\begin{equation}\label{odes2}
\frac{d}{dt}\phi_k(t)=B_k^c(\phi), \qquad \forall k > 0,
\end{equation}
where $B_k^c$ denotes the $k$-th component in the Hermite basis of $B^c$. Namely
\begin{equation}\label{B}
B^c(\phi )= \sum_{k> 0} B_k^c(\phi) H_k^c(x),
\end{equation}
where
\begin{equation}\label{Bk}
B_k^c(\phi)=-\frac{1}{c^2|k|} \sum_{ p\geq  |k|/2}\sum_{\substack{ q \geq 0 \\ |q|<|p| }} (|p| -|q|)  A(p,q,k)\phi_p \phi_q .
\end{equation}



\subsection{Regularity}\label{regularity_subs}

In this section we study the $L^r$-regularity of $B^c$ and its derivatives with respect to the measures $\mu_{\sigma^c,\gamma}$.

\begin{prop}\label{reg_B}
Let $\beta\in \R $ and $\varepsilon >0$; then $B^c\in L^r_{\mu_{\sigma^c,\gamma} } (H^{-\varepsilon}_{\sigma^c}(\R^2) ; H^\beta_{\sigma^c}(\R^2))$ for all $r\geq 1$.
\end{prop}
\begin{proof}
It is enough to prove $\Emgs\| B^c(\phi) \|_{\beta,\sigma^c}^{2r} < +\infty$ for all odd $r\geq 1$.
We have
\begin{align*}
\Emgs \| B^c(\phi)\|_{\beta,\sigma^c}^{2r} &  = \Emgs \left( \sum_{k>0}(1+c|k|)^\beta |B_k^c(\phi)|^2 \right)^r \\
&= \Emgs \sum_{k_1, \ldots, k_r} \prod_{i=1, \ldots, r} (1+c|k_i|)^\beta|B_{k_i}^c(\phi)|^2 \\
& \leq \sum_{k_1, \ldots, k_r} \prod_{i=1, \ldots, r} (1+c|k_i|)^\beta \left( \Emgs |B_{k_i}^c(\phi)|^{2r}\right)^{1/r}\\
&= \left[ \sum_{k>0}(1+c|k|)^\beta \left( \Emgs |B_k^c(\phi)|^{2r}\right)^{1/r}\right]^r.
\end{align*}
From \eqref{Bk} we get
\begin{align*}
\Emgs |B_k^c(\phi)|^{2r} & \leq \left[ \sum_{p,p'}\sum_{q,q'} \frac{1}{c^4|k|^{2}} (|p| -|q|) (|p'| -|q'|)A(p,q,k)A(p',q',k) \left( \Emgs (\phi_p \phi_q \bar \phi_{p'} \bar \phi_{q'})^{r} \right)^{1/r }\right]^r \\
& = \left[ 2  \sum_{p,q}\frac{1}{c^4|k|^{2}} (|p| -|q|)^2 A(p,q,k)^2  \left( \Emgs |\phi_p |^{2r} \right)^{1/r }\left( \Emgs |\phi_q |^{2r} \right)^{1/r }  \right]^r \\
& \leq C(r,\gamma, c)  \left[ \sum_{p,q} \frac{1}{|k|^{2}} \frac{(|p| -|q|)^2}{(1+c|p|)^2(1+c|q|)^2} A(p,q,k)^2   \right]^r, 
\end{align*}
where $C(r,\gamma,c)= \left( \frac{4}{\gamma^2c^4}\right)^r r!^2 $. By property \ref{propA3} of $A(p,q,k)$, we obtain 
\begin{align*}
\Emgs |B_k^c(\phi)|^{2r}& \leq  C(r,\gamma,c)  \left[ \frac{1}{|k|^{2}k!} \sum_{p,q}  \frac{p!q!}{(p+q-1-k)!^2  (q-p-1+k)!^2}  \right]^r <+\infty, \qquad \forall ~k>0.
\end{align*}
We conclude that for all $\beta$,
$$
\Emgs \| B^c(\phi)\|_{\beta,\sigma^c}^{2r} \leq C(r,\gamma,c)  \left[ \sum_{k>0 } \sum_{ |p|\geq |k|/2} \sum_{ q\geq 0 \\ |q|<|p| }\frac{p!q!}{(p+q-1-k)!^2  (q-p-1+k)!^2 k! |k|^{2}(1+c|k|)^{-\beta} }\right]^r < +\infty.
$$
\end{proof}

In particular the field $B^c$ takes values in the Cameron-Martin space $H^2_{\sigma^c}$. 

In Malliavin calculus (c.f.\cite{Ma}), for a functional $F$ defined on an abstract Wiener space $(X,P, H)$, where $P$  and $H$ denote the corresponding Wiener measure and Cameron-Martin space, one defines derivatives along directions $h\in H$ as follows:

$$D_h F(\omega )=\lim_{\epsilon \rightarrow 0} \frac{1}{\epsilon }[F(\omega +h )-F(\omega )],$$
the limit being taken almost everywhere with respect to $P$. Then these derivatives determine a gradient operator which is a linear operator on $H$ and we can use the identification $\nabla F \in H$ by Riesz theorem. If $\nabla F$ is Hilbert-Schmidt we can iterate the procedure and define the second gradient (etc).

\begin{prop}\label{HSreg}
Let $\beta\in \R$ and $\varepsilon >0$; then $\nabla B^c\in L^r_{\mu_{\sigma^c,\gamma}}( H^{-\varepsilon}_{\sigma^c}(\R^2); H.S.(H^2_{\sigma^c}(\R^2); H^\beta_{\sigma^c}(\R^2) ) )$ and \newline $\nabla^2 B^c\in L^r_{\mu_{\sigma^c,\gamma}}( H^{-\varepsilon}_{\sigma^c}(\R^2); H.S.(H^2_{\sigma^c}(\R^2) \otimes H^2_{\sigma^c}(\R^2); H^\beta_{\sigma^c}(\R^2) ) )$ for all $r\geq 1$.
\end{prop}
\begin{proof}
First we compute the Malliavin derivative of $B^c(\phi)$ with respect to the $j$-th order Hermite polynomial, $H_j^c \in H^2_{\sigma^c}$; we have
$$
D_{H_j^c}B^c(\phi)= \sum_{k>0} D_{H_j^c}B_k^c(\phi) H_k^c(x),
$$
where by the definition above
$$
D_{H_j^c}B_k^c(\phi)= \lim_{\varepsilon \to 0^+} \frac{1}{\varepsilon} [B_k^c(\phi+ \varepsilon H_j^c)- B_k^c(\phi)]
$$
and the above limit is taken almost everywhere with respect to $\mu_{\sigma^c,\gamma}$. Therefore
\begin{align*}
D_{H_j^c}B_k^c(\phi)& =-\frac{1}{c^2|k|}\left( \sum_{\substack{q \geq 0 \\ |q|<|j|}} (|j|-|q|)A(j,q,k)\phi_q + \sum_{|p|>|j|} (|p|-|j|)A(p,j,k)\phi_p  \right) \\
& = -\frac{1}{c^2|k|} \sum_{\substack{q \geq 0 }} (|j|-|q|)A(j,q,k)\phi_q,
\end{align*}
in the last equality we relabelled the series in $p$ and used property \ref{propA1} of $A(p,q,k)$ (c.f. Remark \ref{prop_A}). Also we have 
$$
D_{H_i^c} D_{H_j^c} B_k^c(\phi) =-\frac{1}{c^2|k|}(|j|-|i|)A(j,i,k).
$$

We denote by $\{\hat H_k^c(x)\}_k$ the orthonormal basis of $H^2_{\sigma^c}$, that is $\hat H_k^c(x)= \frac{H_k^c(x)}{1+ c|k|}$ for all $k\geq 0$, and we have
$$
\| \nabla B^c(\phi)\|_{H.S.(H^2_{\sigma^c}; H^\beta_{\sigma^c})}^{2r} =\left( \sum_{j\geq 0} \| D_{\hat H_j^c} B^c(\phi)\|_{\beta,\sigma^c}^2 \right)^r= \left(\sum_{j,k}\frac{|D_{H_j^c} B_k^c(\phi)|^2}{(1+c|k|)^{-\beta}(1+c|j|)^2}\right)^r
$$
and
\begin{align*}
\mathbb{E}_{\mu_{\sigma^c,\gamma}} \| \nabla B^c(\phi) \|_{H.S.(H^2_{\sigma^c}; H^\beta_{\sigma^c})}^{2r} & = 
\Emgs \sum_{\substack{j_1, \ldots, j_r \\ k_1, \ldots, k_r}} \prod_{i=0}^r \frac{|D_{ H_{j_i}} B_{k_i}^c(\phi)|^2}{(1+c|k_i|)^{-\beta}(1+c|j_i|)^2} \\
& \leq \left[ \sum_{j,k} \frac{\left( \Emgs|D_{ H_j^c} B^c(\phi)|^{2r}\right)^{1/r}}{(1+c|k|)^{-\beta}(1+c|j|)^2}\right]^r
\end{align*}
where
\begin{align*}
\left( \Emgs|D_{ H_j^c} B^c(\phi)|^{2r}\right)^{1/r}& \leq  \sum_{q \geq 0 } \frac{(|j|-|q|)^2}{c^4|k|^2}A(j,q,k)^2\left( \Emgs |\phi_q|^{2r}\right)^{1/r} \\
&= \frac{2 r!^{1/r}}{\gamma } \sum_{q \geq 0 } \frac{(|j|-|q|)^2}{c^4|k|^2}A(j,q,k)^2\frac{1}{(1+c|q|)^2}.
\end{align*}
By property \ref{propA3} of $A(p,q,k)$ and for every $\beta\in \R$
\begin{align*}
\mathbb{E}_{\mu_{\sigma^c,\gamma}} \| \nabla B^c(\phi) \|_{H.S.(H^2_{\sigma^c}; H^\beta_{\sigma^c})}^{2r}& \leq \left(  \frac{2}{\gamma c^4}\right)^r r! \left[   \sum_{j,k,q}  \frac{(|j|-|q|)^2}{(1+c|k|)^{-\beta}(1+c|j|)^2|k|^2(1+c|q|)^2}A(j,q,k)^2 \right]^r 
< +\infty.
\end{align*}
In particular we have
\begin{equation}\label{nablaB_est}
\mathbb{E}_{\mu_{\sigma^c,\gamma}} \| \nabla B^c(\phi) \|_{H.S.(H^2_{\sigma^c}; H^\beta_{\sigma^c})}^{r} \leq \tilde C(\gamma,c)^r r!^{1/2}
\end{equation}

Similarly, for the second order derivative we have
\begin{align*}
\mathbb{E}_{\mu_{\sigma^c,\gamma}} \| \nabla^2 B^c(\phi) \|_{H.S.(H^2_{\sigma^c}\otimes H^2_{\sigma^c}; H^\beta_{\sigma^c})}^{2r} & \leq \left[ \sum_{i,j,k} \frac{\left( \Emgs|D_{ H_i^c}D_{H_j^c} B^c(\phi)|^{2r}\right)^{1/r}}{(1+c|k|)^{-\beta}(1+c|j|)^2(1+c|i|)^2}\right]^r \\
& =  \left[ \sum_{i,j,k}  \frac{(|j|-|i|)^2A(j,i,k)^2}{c^4|k|^2(1+c|k|)^{-\beta}(1+c|j|)^2(1+c|i|)^2}\right]^r < +\infty.
\end{align*}
\end{proof}



\subsection{The divergence operator}
The divergence on an abstract Wiener space $(X,P,H)$ is the dual of the gradient operator on this space. Namely, for $Z:X\rightarrow H$ the divergence of $Z$, denoted by $\div_P Z$, is such that 

$$E_P (F \div_P Z)=-E_P <Z,\nabla F>$$
for every functional $F$ in $L^2_P (X)$ with $\nabla F \in L^2_P (X;H)$.

For all $n$ we denote by $B^{n,c} $ a Galerkin approximation of $B^c$, that is the projection of $B^c$ on the subspace of $L^2_{\sigma^c}$ generated by $\{H_{\alpha_1}, \ldots, H_{\alpha_{d(n)}}\}$ where $\alpha_1, \ldots, \alpha_{d(n)}$ denote non-negative pairs of $\Z^{2}$. We have
\begin{equation}\label{Galerkin_approxB}
B^{n,c} (\phi)=\sum_{k\in\{\alpha_1, \ldots, \alpha_{d(n)}\}}B^{n,c} _k(\phi) H_k^c(x).
\end{equation}
We denote by $\mu_{\sigma^c,\gamma}^n$ the probability measure given by 
$$
d\mu_{\sigma^c,\gamma}^n(\phi) = \prod_{k\in\{\alpha_1, \ldots, \alpha_{d(n)}\}} \frac{1}{Z^n_{\gamma,k}}e^{-\frac{\gamma}{2}(1+c|k|)^2|\phi_k|^2}d\phi_k,
$$
and by $\eta^n_\gamma$ the Radon-Nikodym density of $d\mu^n_{\sigma^c,\gamma}$ with respect to the Lebesgue measure, $d\lambda^n$, that is $\eta^n_\gamma=d\mu^n_{\sigma^c,\gamma}/d\lambda^n=\frac{1}{Z_\gamma^n}e^{-\frac{\gamma}{2}\sum_{k\in \{\alpha_1, \ldots, \alpha_{d(n)}  \} } (1+c|k|)^2|\phi_k|^2 }$, where $\frac{1}{Z_\gamma^n}=\prod_{k\in \{\alpha_1, \ldots, \alpha_{d(n)}\}} \frac{\gamma (1+c|k|)^2}{2\pi}$.

The divergence of $B^{n,c} $ with respect to the measure $\mu_{\sigma^c,\gamma}^n$ is given by
\begin{align*}
\div_{\mu_{\sigma^c,\gamma}^n} B^{n,c} (\phi) & = \div B^{n,c} (\phi) + < B^{n,c} _k(\phi),  \frac{\nabla \eta^n_\gamma }{\eta^n_\gamma}>_{\C^{d(n)}}.
\end{align*}
On one hand
\begin{align*}
\div B^{n,c} (\varphi) &= \sum_k D_{H_k^c} B^{n,c} _k(\varphi)  =\sum_k \lim_{\varepsilon \to 0^+}\frac{1}{\varepsilon} \left[ B^{n,c} _k(\phi + \varepsilon H_k^c)- B^{n,c} _k(\phi) \right] \\
& = -\sum_{k\in \{\alpha_1, \ldots, \alpha_{d(n)}\}} \frac{1}{c^2|k|} \sum_{ p\geq 0 }(|p|-|k|)A(p,k,k)\phi_p ;
\end{align*} 
on the other
\begin{align*}
< B^{n,c} _k(\phi),  \frac{\nabla \eta^n_\gamma }{\eta^n_\gamma}>_{\C^{d(n)}}& =- \gamma\sum_{k\in \{\alpha_1, \ldots, \alpha_{d(n)}\}} (1+c|k|)^2 B^{n,c} _k(\phi) \bar \phi_k \\
& = \gamma \sum_{k\in \{\alpha_1, \ldots, \alpha_{d(n)}\}} \frac{(1+c|k|)^2}{c^2|k|} \sum_{|p|\geq |k|/2} \sum_{\substack{q\geq 0 \\ |q|<|p|}}(|p|-|q|)A(p,q,k)\phi_p \phi_q \bar \phi_k.
\end{align*}
Therefore,
\begin{align}
\div_{\mu_{\gamma,\sigma^c}^n} B^{n,c} (\phi)& =  -\sum_{k\in \{\alpha_1, \ldots, \alpha_{d(n)}\}}\frac{1}{c^2|k|} \sum_{p\geq 0 }(|p|-|k|)A(p,k,k)\phi_p \\
&+ \gamma \sum_{k\in \{\alpha_1, \ldots, \alpha_{d(n)}\}} \frac{(1+c|k|)^2}{c^2|k|}  \sum_{|p|\geq |k|/2} \sum_{\substack{q\geq 0 \\ |q|<|p|}}(|p|-|q|)A(p,q,k)\phi_p \phi_q \bar \phi_k.
\end{align}

\begin{prop}\label{reg_div}
Let $\beta\in\R$ and $\varepsilon>0$, then for all $r\geq 1$ we have $\div_{\mu_{\sigma^c, \gamma}} B^c\in L^r_{\mu_{\sigma^c, \gamma}}(H^{-\varepsilon}_{\sigma^c}; \R)$.
\end{prop}
\begin{proof}
We show that
$$
\Emgs | \div_{\mu_{\sigma^c, \gamma}} B^c |^{2r}<+\infty
$$ 
for all odd $r\geq 1$. We have
\begin{align*}
\left[ \Emgs | \div_{\mu_{\sigma^c, \gamma}} B^c |^{2r} \right]^{1/2r} & \leq  \left[ \Emgs  \Big| \sum_{k}  \frac{1}{c^2|k|}  \sum_{p\geq 0  }(|p|-|k|)A(p,k,k)\phi_p \Big|^{2r} \right]^{1/2r} \\
&  +  \left[ \Emgs \Big|  \sum_{k}  \frac{(1+c|k|)^2}{c^2|k|}  \sum_{|p|\geq |k|/2} \sum_{\substack{q\geq 0 \\ |q|<|p|}}(|p|-|q|)A(p,q,k)\phi_p \phi_q \bar \phi_k \Big|^{2r}  \right]^{1/2r} \\
& \leq  \left[ \sum_{k,p}\frac{1}{c^4|k|^2} (|p|-|k|)^2A(p,k,k)^2 ( \Emgs |\phi_p|^{2r} )^{1/r} \right]^{1/2}  \\
&  +  \left[    \sum_{k,p,q}   \frac{(1+c|k|)^4}{c^4|k|^2} (|p|-|q|)^2A(p,q,k)^2 ( \Emgs ( \phi_p \phi_q \bar \phi_k)^{2r} )^{1/r} \right]^{1/2}\\
& \leq  \frac{2}{\gamma} r!^{\frac{1}{2r}} \left[ \sum_{k,p}\frac{1}{c^4|k|^2 } \frac{(|p|-|k|)^2}{(1+c|p|)^2}A(p,k,k)^2  \right]^{1/2}  \\
&  +  \frac{2}{\gamma}^{3/2}(3r)!^{\frac{1}{2r}}\frac{1}{c^2} \left[    \sum_{k,p,q}   \frac{(1+c|k|)^4}{|k|^2} \frac{(|p|-|q|)^2}{(1+c|p|)^2(1+c|q|)^2(1+c|k|)^2}A(p,q,k)^2  \right]^{1/2}  < +\infty.
\end{align*}
In particular we get
\begin{equation}\label{div_est}
\left[ \Emgs | \div_{\mu_{\sigma^c, \gamma}} B^c |^{r} \right]^{1/r}  \leq C(\gamma,c) r!^{\frac{1}{2r}}.
\end{equation}
\end{proof}



\section{Existence and quasi-invariance}

In this section, we prove that there exists a flow for the vector field $B^c$ defined almost everywhere with respect to each probability measure $\mu_{\sigma^c,\gamma}$.  Moreover, we show that the probability measures $\mu_{\sigma^c,\gamma}$ are quasi-invariant with respect to these flows. 

The proof of these facts will follow from a result by A. S. Ustunel, Theorem 5.3.1 of \cite{U}. This theorem gives some exponential integrability conditions on the vector field that ensure existence and quasi-invariance, generalizing a previous result by A. B. Cruzeiro \cite{C3}. Both results hold for vector fields on Wiener spaces taking values in the Cameron-Martin spaces, thus below we fix $\beta=2$. 

Recall that, if  $\nu$ is a measure defined on a probability space $(\Omega, \mathcal{F}, \mathbb{P})$ and $T: \supp_\nu \to \supp_\nu$, we say that $\nu$ is quasi-invariant under $T$ if $T*\nu << \nu $.

\begin{thm}\label{ex}
Let $\beta=2$ and $\varepsilon >0$, then $B^c:  H^{-\varepsilon}_{\sigma^c} \to H^2_{\sigma^c}$ is such that there exists an almost surely unique flow for $B^c$ defined by 
\begin{equation}
U_t^c(\phi)= \phi + \int_0^t B^c(U_s^c(\phi))ds \qquad \mu_{\sigma^c,\gamma}-a.e. ~\phi \in H^{-\varepsilon}_{\sigma^c} \cap L^p_{loc},\quad \forall t\in \R.
\end{equation}
Moreover, the measure $\mu_{\sigma^c,\gamma}$ is quasi-invariant under $U^B_t$ and 
\begin{equation}\label{density}
k_t(\phi) = \exp\left( \int_0^t \div_{\mu_{\sigma^c,\gamma}} B^c(U_{-s}^c(\phi)) ds \right)
\end{equation}
is the corresponding Radon-Nikodym density, defined by $k_t:=\frac{d U_t^c*\mu_{\sigma^c,\gamma}}{d \mu_{\sigma^c,\gamma}}$. 
We have $k_t \in L^r_{\mu_{\sigma^c,\gamma}}$, for all $r\geq 1$.
\end{thm}

\begin{proof}
We know from Propositions \ref{reg_B} to \ref{reg_div} that $B^c\in L^r_{\mu_{\sigma^c,\gamma} } (H^{-\varepsilon}_{\sigma^c}(\R^2) ; H^\beta_{\sigma^c}(\R^2))$; \\ 
$\nabla B^c\in L^r_{\mu_{\sigma^c,\gamma}}( H^{-\varepsilon}_{\sigma^c}(\R^2); H.S.(H^2_{\sigma^c}(\R^2); H^\beta_{\sigma^c}(\R^2) ) )$; and that $\div_{\mu_{\sigma^c, \gamma}} B^c\in L^r_{\mu_{\sigma^c, \gamma}}(H^{-\varepsilon}_{\sigma^c}; \R)$, for all $r\geq 1$. In order to apply Ustunel's result we only have to prove that  for any given $t\in \R$, there exists a positive $\lambda$ such that 
\begin{equation}
\int_0^t \Emgs \left[ \exp(\lambda |\div_{\mu_{\sigma^c,\gamma}} B^c(\phi) |) + \exp(\lambda \| \nabla B^c(\phi)\| )\right] < +\infty, 
\end{equation}
where $\| \nabla B^c(\phi)\|$ is the operator norm given by $\sup_{\substack{ h\in H^2_{\sigma^c} \\ |h|\leq 1 }}\| D_h B^c(\phi) \|_{2,\sigma^c}$.

We use estimatives \eqref{div_est} and \eqref{nablaB_est} to get respectively
\begin{align*}
\int_0^t \Emgs \left[ \exp(\lambda |\div_{\mu_{\sigma^c,\gamma}} B^c(\phi) |)\right] & = \int_0^t  \sum_{j \geq 0} \frac{\lambda^j}{j!} \Emgs |\div_{\mu_{\sigma^c,\gamma}} B^c(\phi) |^j \\
& \leq  |t|  \sum_{j \geq 0} \frac{(\lambda C(\gamma,c))^j}{\sqrt{j!}} < +\infty, \quad \forall \lambda>0,
\end{align*}
and 
\begin{align*}
\int_0^t \Emgs \left[  \exp(\lambda \| \nabla B^c(\phi)\| )\right] & \leq \int_0^t \Emgs \left[  \exp(\lambda \| \nabla B^c(\phi)\|_{H.S.} )\right] \\
&= \int_0^t  \sum_{j \geq 0} \frac{\lambda^j}{j!} \Emgs  \| \nabla B^c(\phi)\|_{H.S.}^j \\
& \leq |t|  \sum_{j \geq 0} \frac{(\lambda \tilde C(\gamma,c))^j}{\sqrt{j!}} < +\infty, \quad \forall \lambda>0.
\end{align*}

We conclude the proof since all the hypothesis of Theorem 5.3.1 of \cite{U} are satisfied. In the work \cite{C3}, under these assumptions, it is proved that $k_t \in L^r_{\mu_{\sigma^c,\gamma}}$ for all $r\geq 1$. 

\end{proof}

Finally, we recover the velocity $\tilde u$ and the pressure $p$. On one hand
\begin{equation}\label{u^c}
\tilde u= \sigma^c \nabla^\perp U^c_t(\phi), \qquad \mu_{\sigma^c,\gamma}-a.e. ~\phi \in H^{-\varepsilon}_{\sigma^c}(\R^2) \cap L^p_{loc}(\R^2),
\end{equation}
on the other, by taking the divergence of equation \eqref{Euler1}, we obtain
$$
(L^c-2cI)p= -\div \Tr |\nabla \tilde u  |^2.
$$
The operator $(L^c-2cI)$ is invertible since the value $2c$ doesn't belong to the spectrum of $L^c$. Moreover, we computed in Subsection \ref{weak_sol} the integral kernel of the inverse of $L^c$, see equations \eqref{G_L^r}-\eqref{GLr}, from this and by a perturbative argument it is possible to get the integral kernel of $(L^c-2cI)^{-1}$. 
Hence we have
\begin{equation}\label{p}
p= - (L^c-2cI)^{-1}\div\Tr|\nabla \nabla^\perp U_t^c(\phi)|^2, \qquad \mu_{\sigma^c,\gamma}-a.e. ~\phi \in H^{-\varepsilon}_{\sigma^c}(\R^2) \cap L^p_{loc}(\R^2).
\end{equation}

Last we observe that in the limit when the parameter $c$  tends to zero, equations \eqref{Euler1} converge to the ``standard" incompressible Euler equations, thus we can formally  look at \eqref{u^c} and \eqref{p} as approximations of the solutions for these equations. However, we cannot rigorously consider such limit since the measures $\mu_{\sigma^c,\gamma} $ for each  $c$ are mutually singular.

\subsection*{Acknowledgements}\noindent
The authors thank Prof. Nikolay Tzvetkov for very useful discussions. They acknowledge the support of FCT project PTDC/MAT-STA/0975/2014. The second author was also funded by the LisMath fellowship PD/BD/52641/2014, FCT, Portugal.


\bibliographystyle{plain}

\end{document}